\documentclass[12pt,a4paper]{article}
\usepackage[utf8]{inputenc}
\usepackage[english]{babel}
\usepackage{amsmath,amssymb,amsfonts}%
\usepackage{amsthm}%
\usepackage{mathrsfs}%
\usepackage{dsfont}
\usepackage[width=6.50in, height=8.00in, left=1.00in, top=1.00in]{geometry}

\DeclareMathOperator{\ReC}{\mathrm{Re}}
\DeclareMathOperator{\sign}{\mathrm{sign}}

\theoremstyle{plain}
\newtheorem{theorem}{Theorem}
\newtheorem{lemma}{Lemma}
\newtheorem{remark}{Remark}

\author{Mykhailo Osypchuk}

\title{Bilateral estimates of some pseudo-derivatives of the~transition probability density of an~isotropic $\alpha$-stable stochastic process}

\date{\today}

\begin{document}
	
\maketitle
	
\begin{abstract}In the paper, the transition probability density of isotropic $\alpha$-stable stochastic process in a finite dimensional Euclidean space is considered. The results of applying pseudo differential operators with respect spatial variables to this function are estimated from the both side: above and below. Operators in the consideration are defined by the symbols $|\lambda|^\varkappa$ and $\lambda|\lambda|^{\varkappa-1}$, where $\varkappa$ is some constant. The first operator with negative sign is fractional Laplacian and the second one multiplied by imaginary unit is fractional gradient.
\end{abstract}


	
\section{Introduction}

\label{s1}
Let us define some constants $\alpha\in(0,2)$, $c>0$ and consider an isotropic $\alpha$-stable stochastic process $(x(t))_{t\ge0}$ in $d$-dimensional Euclidean space $\mathbb{R}^d$.  As usual, we denote by $(\cdot,\cdot)$ the inner product and by $|\cdot|$ the norm in $\mathbb{R}^d$ (we use the last notation for denoting the absolute value of a real number and the module of a complex number). The process $(x(t))_{t\ge0}$ is a strong Markov process with transition probability density $(g(t,x,y))_{t>0, x\in\mathbb{R}^d, y\in\mathbb{R}^d}$  given by the following equality
\begin{equation}
\label{f7}
g(t,x,y)=\frac{1}{(2\pi)^d}\int_{\mathbb{R}^d}\exp\{i(x-y,\lambda)-ct|\lambda|^\alpha\}\mathrm{d}\lambda.
\end{equation}
The function $(g(t,x,y))_{t>0, x\in\mathbb{R}^d, y\in\mathbb{R}^d}$ is the Green's function (or a fundamental solution) of the following pseudo differential equation
$
\frac{\partial}{\partial t}u(t,x)=-c(-\Delta)^\frac{\alpha}{2} u(t,\cdot)(x).	
$
 Here and in what follows $Df(\cdot,y)(x)$ denotes the operator $D$ action on the function $f$ with respect to the first variable in the point $x$. In the case of $\alpha=2$ the process $(x(t))_{t\ge0}$ is the Brownian motion, function \eqref{f7} has an explicit form and all calculation will be of a different type. We do not consider this case.

According to the Blumenthal and Getoor's results (see \cite{BlumGet}), we have the next estimations of the function $(g(t,x,y))_{t>0, x\in\mathbb{R}^d, y\in\mathbb{R}^d}$
\begin{equation}
\label{f22}
\frac{G_1\, t}{(t^\frac{1}{\alpha}+|x-y|)^{d+\alpha}}\le g(t,x,y)\le\frac{G_2\, t}{(t^\frac{1}{\alpha}+|x-y|)^{d+\alpha}},\quad t>0,\ x\in\mathbb{R}^d, y\in\mathbb{R}^d,
\end{equation}
where $G_1$ and $G_2$ are some positive constants.

Let us define the following operators by their symbols: the operator $\Delta_\varkappa$ has the symbol $(|\lambda|^\varkappa)_{\lambda\in\mathbb{R}^d}$, and $\nabla_\varkappa$ has the symbol  $(\lambda|\lambda|^{\varkappa-1})_{\lambda\in\mathbb{R}^d}$. Here $\varkappa$ is some constant. So far, we do not limit the value of $\varkappa$.
\begin{remark}
	The operator $-\Delta_\varkappa$ is the  fractional Laplacian of the order $\varkappa$ and the operator $i\nabla_\varkappa$ is the fractional gradient  of the order $\varkappa$.
\end{remark}

For $\varkappa>0$ and $\alpha\ge1$ it is well-known the following estimation ($M>0$ is some constant)
\begin{equation}
\label{f23}
|D^{(\varkappa)} g(t,\cdot,y)(x)|\le \frac{M}{(t^\frac{1}{\alpha}+|x-y|)^{d+\varkappa}},\quad t>0,\ x\in\mathbb{R}^d, y\in\mathbb{R}^d,
\end{equation}
where $D^{(\varkappa)}$ means any pseudo differential operator with smooth enough and homogeneous of the degree $\varkappa$ symbol. If $\varkappa$ is additionally integer we have more accurate estimation 
\begin{equation}
	\label{f23_1}
	|D^{(\varkappa)} g(t,\cdot,y)(x)|\le \frac{M t}{(t^\frac{1}{\alpha}+|x-y|)^{d+\alpha+\varkappa}},\quad t>0,\ x\in\mathbb{R}^d, y\in\mathbb{R}^d.
\end{equation}
For details, see \cite[Ch. 4]{EidIvKoch}.

The goal of this paper is to obtain the similar to \eqref{f22} estimations for $\nabla_\varkappa g(t,\cdot,y)$ and $\Delta_\varkappa g(t,\cdot,y)$ ($t>0$ and $y\in\mathbb{R}^d$ are fixed) with as wide set of parameter $\varkappa$ values as possible.

Our main results are presented in Section \ref{s2}. There we give exact asymptotics, if $|x|\to+\infty$, for $\Delta_\varkappa g(1,\cdot,0)(x)$ and  $\nabla_\varkappa g(1,\cdot,0)(x)$ with $\varkappa>-(d\wedge2)$ or  $\varkappa>1-(d\wedge2)$, respectively. Using the mentioned asymptotics we establish estimations of $\Delta_\varkappa g(t,\cdot,y)(x)$ and  $\nabla_\varkappa g(t,\cdot,y)(x)$ for all $t>0$, $x\in\mathbb{R}^d$ and $y\in\mathbb{R}^d$.

\section{One auxiliary lemma}
The following statement will be used to construct asymptotic and estimations below in the paper. But it can have an independent meaning. Our proof method is similar to Blumenthal and Getoor's proof in a simpler case (see \cite{BlumGet}).
\begin{lemma}\label{lem1}
	For any pair of real numbers $\nu$ and $\mu$ satisfied the assumption $\nu>|\mu|-1$ and a real number $\alpha>0$ the following equality
	\begin{equation}
		\label{f1_1}
		\lim_{r\to+\infty}\int_0^{+\infty} t^\nu e^{-\left(\frac{t}{r}\right)^\alpha} J_\mu(t)\mathrm{d}t =\frac{2^\nu}{\pi}\Gamma\left(\frac{\nu-\mu+1}{2}\right) \Gamma\left(\frac{\nu+\mu+1}{2}\right) \cos\frac{\pi(\nu-\mu)}{2},
	\end{equation}
	in which $J_\mu$ denotes the Bessel function of the first kind of order $\mu$, holds true. 	
\end{lemma}
\begin{proof}
	As usual, denote by $K_\mu$ the modified Bessel function of the second kind of the order $\mu$ and by $H_\mu^{(1)}$ the first of the Hankel functions of the order $\mu$ (also named by the Bessel function of the third kind). Use the known relations between Bessel functions, i.e.
	\[
	J_\mu(z)=\ReC H_\mu^{(1)}(z),\quad K_\mu(z)=\frac{\pi i}{2}e^{\mu\frac{\pi}{2}i}H_\mu^{(1)}(iz),
	\] 
	which are fulfilled for all complex numbers $\mu$ and $z$ (with $-\pi<\arg z<\pi/2$ in second equality). 
	
	After a change of the variable in the integral from formula \eqref{f1_1} we obtain the following equality (we put $t=is$)
	\begin{equation}\label{f1_3}
		\int_0^{+\infty} t^\nu e^{-\left(\frac{t}{r}\right)^\alpha} J_\mu(t)\mathrm{d}t = -\frac{2}{\pi}\ReC\left(e^{(\nu-\mu)\frac{\pi}{2}i} \int_{-\infty i}^0 s^\nu e^{-\left(\frac{s}{r}\right)^\alpha e^{\alpha\frac{\pi}{2}i}}K_\mu(s)\mathrm{d}s\right). 
	\end{equation}
	Let us denote the last integral by $K(r)$. Then we can write down the following equality
	\begin{equation}\label{f1_2}
		K(r)=-\int_{l_\beta}z^\nu e^{-\left(\frac{z}{r}\right)^\alpha e^{\alpha\frac{\pi}{2}i}}K_\mu(z)\mathrm{d}z,
	\end{equation}
	where $l_\beta=\{z=se^{i\beta}:s\ge0\}$ with some $\beta\in\left(-\frac{\pi}{2},\frac{\pi}{2}\left(\frac{1}{\alpha}-1\right)\wedge 0\right)$.	Here we are using that for $\gamma_R=\{z=Re^{i\varphi}:-\frac{\pi}{2}\le\varphi\le\beta\}$ with any big enough $R>0$, the following relations
	\[
	\left|\int_{\gamma_R}z^\nu e^{-\left(\frac{z}{r}\right)^\alpha e^{\alpha\frac{\pi}{2}i}}K_\mu(z)\mathrm{d}z\right|\le \left(\beta+\frac{\pi}{2}\right)R^{\nu+1}e^{-\left(\frac{R}{r}\right)^\alpha\cos(\beta+\frac{\pi}{2})\alpha}\sqrt\frac{2\pi}{R}\to0, \mbox{ as }R\to+\infty,
	\]
	are fulfilled, since, as it is well-known, $K_\mu(z)\sim\sqrt{\frac{\pi}{2z}}e^{-z}$, as $z\to\infty$.
	
	Now, using that there can be passing to the limit as $r\to+\infty$ in the integral from \eqref{f1_2}. Really, the integrand in \eqref{f1_2} is estimated by the integrated expressions $\sqrt{2\pi}s^{\nu-\frac{1}{2}}e^{-s\cos\beta}$ for a big enough $s=|z|$ and $s^{\nu-|\mu|}\Gamma(|\mu|)2^{|\mu|}$ for small $s=|z|$. Here we are using the well-known relation $K_\mu(z)\sim z^{-|\mu|}\Gamma(|\mu|)2^{|\mu|-1}$, as $z\to0$.
	
	Since for $\hat{\gamma}_R=\{z=Re^{i\varphi}:\beta\le\varphi\le0\}$ with some big enough $R>0$, we have
	\[
	\left|\int_{\hat{\gamma}_R}z^\nu K_\mu(z)\mathrm{d}z\right|\le-\beta R^{\nu-\frac{1}{2}}\sqrt{2\pi}e^{-R\cos\beta}\to0,\quad R\to+\infty,
	\]
	the following equalities
	\begin{equation}\label{f1_4}
		\lim_{r\to+\infty}K(r)=-\int_{l_\beta}z^\nu K_\mu(z)\mathrm{d}z=-\int_0^{+\infty}t^\nu K_\mu(t)\mathrm{d}t
	\end{equation}
	hold true.
	The last integral in \eqref{f1_4} can be calculated if $\nu>|\mu|-1$. This can be checked by using mathematical software. It equals to 
	$2^\nu\Gamma\left(\frac{\nu-\mu+1}{2}\right) \Gamma\left(\frac{\nu+\mu+1}{2}\right)$. 
	Finally, using this fact and formulae \eqref{f1_3} and \eqref{f1_4}, we obtain the lemma statement.
\end{proof}

\section{Asymptotics and estimations}

\label{s2}
According to presentation \eqref{f7} of the function $(g(t,x,y))_{t>0, x\in\mathbb{R}^d, y\in\mathbb{R}^d}$ we can write down the following relations ($t>0$, $x\in\mathbb{R}^d$, $y\in\mathbb{R}^d$)
\[
\Delta_\varkappa g(t,\cdot,y)(x)=\frac{1}{(2\pi)^d}\int_{\mathbb{R}^d}|\lambda|^\varkappa\exp\{i(x-y,\lambda)-ct|\lambda|^\alpha\}\mathrm{d}\lambda,
\]
\[
\nabla_\varkappa g(t,\cdot,y)(x)=\frac{1}{(2\pi)^d}\int_{\mathbb{R}^d}\lambda |\lambda|^{\varkappa-1}\exp\{i(x-y,\lambda)-ct|\lambda|^\alpha\}\mathrm{d}\lambda,
\]
if the corresponding integrals exist. It is easy to see that the constant $\varkappa$ must be greater than $-d$. We consider this condition fulfilled.

To simplify the entries, we introduce the following notations
\begin{equation}
\label{f9_1}
D(\varkappa,x)=\frac{1}{(2\pi)^d}\int_{\mathbb{R}^d}|\lambda|^\varkappa\exp\{i(x,\lambda)-|\lambda|^\alpha\}\mathrm{d}\lambda,
\end{equation}
\begin{equation}
\label{f9_2}
N(\varkappa,x)=\frac{1}{(2\pi)^d}\int_{\mathbb{R}^d}\lambda|\lambda|^{\varkappa-1}\exp\{i(x,\lambda)-|\lambda|^\alpha\}\mathrm{d}\lambda.
\end{equation}
So, for all $t>0$, $x\in\mathbb{R}^d$, $y\in\mathbb{R}^d$, and $\varkappa>-d$, we have 
\begin{equation}
\label{f10}
\Delta_\varkappa g(t,\cdot,y)(x)=(ct)^{-\frac{d+\varkappa}{\alpha}}D(\varkappa,(ct)^{-\frac{1}{\alpha}}(x-y)), 
\end{equation}
\begin{equation}
\label{f11}
\nabla_\varkappa g(t,\cdot,y)(x)=(ct)^{-\frac{d+\varkappa}{\alpha}}N(\varkappa,(ct)^{-\frac{1}{\alpha}}(x-y)).
\end{equation}

The following statement indicates the boundedness of the functions \eqref{f10}, \eqref{f11}.
\begin{theorem}
	For any $\varkappa>-d$, there exists a constant $K>0$, such that for all $t>0$, $x\in\mathbb{R}^d$, and $y\in\mathbb{R}^d$ the following estimation
	\begin{equation}
	\label{f21}
	|\Delta_\varkappa g(t,\cdot,y)(x)|+|\nabla_\varkappa g(t,\cdot,y)(x)|\le K\,t^{-\frac{d+\varkappa}{\alpha}}
	\end{equation}
	holds true.
\end{theorem}	
\begin{proof}
	As it follows from \eqref{f9_1} and \eqref{f9_2}, the functions $(D(\varkappa,x))_{x\in\mathbb{R}^d}$ and $(N(\varkappa,x))_{x\in\mathbb{R}^d}$ are bounded for every $\varkappa>-d$. Using \eqref{f10} and \eqref{f11}, we obtain the theorem statement.
\end{proof}

It is evident that $N(\varkappa,x)=-i\nabla D(\varkappa-1,\cdot)(x)$ for all $x\in\mathbb{R}^d$ and $\varkappa>-d+1$.  It is clear that the functions $(D(\varkappa,x))_{x\in\mathbb{R}^d}$ and $(N(\varkappa,x))_{x\in\mathbb{R}^d}$ are bounded on $\mathbb{R}^d$ for all $\varkappa>-d$. The following statement defines the asymptotic of the function $(D(\varkappa,x))_{x\in\mathbb{R}^d}$ if $|x|\to+\infty$ for fixed $\varkappa$.
\begin{theorem}
	\label{th1}
	For $\varkappa>-(d\wedge2)$ the following relation
	\begin{equation}
	\label{f12}
	\lim_{|x|\to+\infty}|x|^{d+\varkappa}D(\varkappa,x)=\frac{2^\varkappa}{\pi^{\frac{d}{2}+1}} \Gamma\left(\frac{d+\varkappa}{2}\right)\Gamma\left(\frac{\varkappa}{2}+1\right) \cos\frac{\varkappa+1}{2}\pi
	\end{equation}
	is fulfilled.
\end{theorem}
\begin{proof}
According to the function $D(\varkappa,\cdot)$ is Fourier transform of the radial function, we can write down (see \cite[Ch II, \S 7, p. 69]{BochChandr})
the following representation
\begin{equation}\label{f12_2}
D(\varkappa,x)=(2\pi)^{-\frac{d}{2}}|x|^{-d-\varkappa}I(d,\varkappa,|x|),
\end{equation}
where $I(d,\varkappa,r):=\int_0^{+\infty} t^{\frac{d}{2}+\varkappa} e^{-\left(\frac{t}{r}\right)^\alpha} J_{\frac{d}{2}-1}(t)\mathrm{d}t$.

It is easy to verify that the assumptions of Lemma \ref{lem1} are fulfilled for $\nu=\frac{d}{2}+\varkappa$ and $\mu=\frac{d}{2}-1$ if $\varkappa>-(d\wedge2)$. Then we obtain
\[
D(\varkappa,x)\sim|x|^{-d-\varkappa}\frac{2^\varkappa}{\pi^{\frac{d}{2}+1}} \Gamma\left(\frac{\varkappa}{2}+1\right)\Gamma\left(\frac{d+\varkappa}{2}\right) \cos\frac{\varkappa+1}{2}\pi,\quad |x|\to+\infty,
\]
if $\varkappa>-(d\wedge2)$. The theorem is proved.
\end{proof}

It is evident that this asymptotic is useless if $\varkappa$ is an even integer. Therefore, we have to consider this case separately. The corresponding statement is as follows.
\begin{theorem}\label{th2}
	If $\varkappa\ge0$ and it is even integer, the relation
	\begin{equation}
	\label{f12_1}
	\lim_{|x|\to+\infty}|x|^{d+\alpha+\varkappa}D(\varkappa,x)=\frac{(\alpha-\varkappa) 2^{\varkappa+\alpha-1}}{\pi^{\frac{d}{2}+1}} \Gamma\left(\frac{\alpha+\varkappa}{2}\right) \Gamma\left(\frac{d+\alpha+\varkappa}{2}\right) 
	\cos\frac{\alpha+\varkappa-1}{2}\pi
	\end{equation}
	is fulfilled.	
\end{theorem}
\begin{proof}
Using the well-known relation (see, for example, \cite{Erdeyi})
$J_\mu^{'}(t)=-\frac{\mu}{t}J_\mu(t)+J_{\mu-1}(t)$ if $t\neq0$, we can write down the following equality
\[
I(d,\varkappa,r)=\int_0^{+\infty} t^{\frac{d}{2}+\varkappa} e^{-\left(\frac{t}{r}\right)^\alpha} J_{\frac{d}{2}}^{'}(t)\mathrm{d}t+ \frac{d}{2}I(d+2,\varkappa-2,r),
\]
that holds true if $\varkappa\neq0$ additionally. The case of $\varkappa=0$ was considered in \cite{BlumGet}. Excluding this case and integrating by part, we obtain the following relation
\[
I(d,\varkappa,r)=\frac{\alpha}{r^\alpha}I(d+2,\varkappa+\alpha-2,r)-\varkappa I(d+2,\varkappa-2,r).
\]
Note that this equality is also true if $\varkappa=0$ or if it is not integer.
It is easy exercise to verify, for example by using Lopital's rule, that
$$
\lim_{r\to+\infty}r^\alpha I(d+2,\varkappa-2,r)=\lim_{r\to+\infty}I(d+2,\varkappa+\alpha-2,r)
$$ 
for any $\varkappa>0$. In addition, using Lemma \ref{lem1} we obtain
$$
\lim_{r\to+\infty}I(d+2,\varkappa-2,r)=\frac{2^{\frac{d}{2}+\varkappa+\alpha-1}}{\pi} \Gamma\left(\frac{\alpha+\varkappa}{2}\right) \Gamma\left(\frac{d+\alpha+\varkappa}{2}\right) 
\cos\frac{\alpha+\varkappa-1}{2}\pi.
$$ 
Therefore using relation \eqref{f12_2} we obtain the theorem statement.	
\end{proof}

\begin{remark}
If $\varkappa=0$ relation \eqref{f12_1} leads us to the corresponding one obtained in \cite{BlumGet}.	
\end{remark}

The next statement provides the asymptotic of $N(\varkappa,x)$ if $|x|\to+\infty$. 

\begin{theorem}
	\label{th3}
	If $\varkappa>1-(d\wedge2)$ and $\varkappa$ is not an odd integer, for any unit vector $\nu\in\mathbb{R}^d$ the following relation
	\[
	(N(\varkappa,x),\nu)\sim 
	i\frac{2^{\varkappa}}{\pi^{\frac{d}{2}+1}}
	\Gamma\left(\frac{\varkappa+1}{2}\right)
	\Gamma\left(\frac{d+\varkappa+1}{2}\right)
	\cos\frac{\varkappa\pi}{2}
	\frac{(x,\nu)}{|x|^{d+\varkappa+1}},\quad |x|\to+\infty
	\]
	holds true.
\end{theorem}
\begin{proof}
	As it was mentioned above, we have that $N(\varkappa,x)=-i\nabla D(\varkappa-1,\cdot)(x)$ for all $x\in\mathbb{R}^d$ and $\varkappa>-d+1$. Using expression \eqref{f9_1} and integrating by part one can easily obtain the following equality
	\[
	(N(\varkappa,x),\nu)=-i\frac{(x,\nu)}{|x|^{2}}\left(\alpha D(\alpha+\varkappa-1,x)-(d+\varkappa-1)D(\varkappa-1,x)\right),
	\]
	which is fulfilled for all $x\in\mathbb{R}^d$ and $\varkappa>-d+1$. This equality and the statement of Theorem~\ref{th1} lead us to the relations (if $|x|\to+\infty$)
	\[
	\begin{split}
	(N(\varkappa,x),\nu)\sim-i\frac{(x,\nu)}{|x|^{2}}\left(\alpha D_{\alpha+\varkappa-1}|x|^{-d-\alpha-\varkappa+1}-(d+\varkappa-1)D_{\varkappa-1}|x|^{-d-\varkappa+1}\right)\\
	\sim i(d+\varkappa-1)D_{\varkappa-1}|x|^{-d-\varkappa-1}(x,\nu),	
	\end{split}\]
	that hold for all $\varkappa>1-(d\wedge2)$, where by $D_\varkappa$ we denote the value of the limit presented by formula \eqref{f12}. The theorem statement is already obvious.
\end{proof}

A special case is that the number $\varkappa$ is an odd integer. The corresponding statement is as follows.
\begin{theorem}\label{th4}
If an odd integer $\varkappa\ge1$, then
\[\begin{split}
(N(\varkappa&,x),\nu)\sim 
-i\frac{(d+\varkappa-1)(\alpha-\varkappa+1)2^{\varkappa+\alpha-2}}{\pi^{\frac{d}{2}+1}}\\
&\times\Gamma\left(\frac{\alpha+\varkappa-1}{2}\right)
\Gamma\left(\frac{d+\alpha+\varkappa-1}{2}\right)
\cos\frac{(\alpha+\varkappa)\pi}{2}
\frac{(x,\nu)}{|x|^{d+\alpha+\varkappa+1}},\quad |x|\to+\infty,
\end{split}
\]
for any unit vector $\nu\in\mathbb{R}^d$.
\end{theorem}
\begin{proof}
The proof is similar to Theorem \ref{th3} proof if use the statement of Theorem \ref{th2} instead of Theorem \ref{th1}.
\end{proof}

Using the statements of Theorems \ref{th1}~--- \ref{th4} we obtain bilateral estimations of the functions $(\Delta_\varkappa g(t,\cdot,y)(x))_{t>0,x\in\mathbb{R}^d,y\in\mathbb{R}^d}$ and $(\nabla_\varkappa g(t,\cdot,y)(x))_{t>0,x\in\mathbb{R}^d,y\in\mathbb{R}^d}$.

Let us introduce the notations:
\[
D_\varkappa:=\frac{2^\varkappa}{\pi^{\frac{d}{2}+1}} \Gamma\left(\frac{d+\varkappa}{2}\right)\Gamma\left(\frac{\varkappa}{2}+1\right) \cos\frac{\varkappa+1}{2}\pi
\] 
for $\varkappa>-(d\wedge2)$ except of  an even integer;
\[
D_{\varkappa,\alpha}:=\frac{(\alpha-\varkappa) 2^{\varkappa+\alpha-1}}{\pi^{\frac{d}{2}+1}} \Gamma\left(\frac{\alpha+\varkappa}{2}\right) \Gamma\left(\frac{d+\alpha+\varkappa}{2}\right) 
	\cos\frac{\alpha+\varkappa-1}{2}\pi
\]
for $\alpha\in(0,2)$ and an even integer $\varkappa\ge0$;
\[
N_\varkappa:=-\frac{2^{\varkappa}}{\pi^{\frac{d}{2}+1}}
	\Gamma\left(\frac{\varkappa+1}{2}\right)
	\Gamma\left(\frac{d+\varkappa+1}{2}\right)
	\cos\frac{\varkappa\pi}{2}
\]
for  $\varkappa>1-(d\wedge2)$ except of an odd integer;
\[
N_{\varkappa,\alpha}:=\frac{(d+\varkappa-1)(\alpha-\varkappa+1)2^{\varkappa+\alpha-2}}{\pi^{\frac{d}{2}+1}}\Gamma\left(\frac{\alpha+\varkappa-1}{2}\right)
	\Gamma\left(\frac{d+\alpha+\varkappa-1}{2}\right)
	\cos\frac{(\alpha+\varkappa)\pi}{2}
\]
for $\alpha\in(0,2)$ and an odd integer $\varkappa\ge1$.

The following statement is simple consequence of Theorems \ref{th1}~--- \ref{th4}. 
\begin{theorem}\label{corol}
	For any fixed $0<\varepsilon<1$, there exists a constant $K>0$, such that for all $t>0$, $x\in\mathbb{R}^d$, and $y\in\mathbb{R}^d$ satisfying the inequality $|x-y|\ge K t^{\frac{1}{\alpha}}$, there are fulfilled the following inequalities
	\begin{equation}
	\label{f19}
	(1-\varepsilon\sign D_\varkappa)D_\varkappa|x-y|^{-d-\varkappa}\le\Delta_\varkappa g(t,\cdot,y)(x)\le(1+\varepsilon\sign D_\varkappa)D_\varkappa|x-y|^{-d-\varkappa}
	\end{equation}for $\varkappa>-(d\wedge2)$, which is not even integer;
	\begin{equation}
	\label{f20}
	\begin{split}
	\big(1-\varepsilon\sign(N_\varkappa(x-y,\nu))\big)N_\varkappa\frac{(x-y,\nu)}{|x-y|^{d+\varkappa+1}}\le i(\nabla_\varkappa g(t,\cdot,y)(x),\nu)\\
	\le\big(1+\varepsilon\sign(N_\varkappa(x-y,\nu))\big)N_\varkappa\frac{(x-y,\nu)}{|x-y|^{d+\varkappa+1}}
\end{split}
		\end{equation}
for $\varkappa>1-(d\wedge2)$, which is not odd integer, and any $\nu\in\mathbb{R}^d$;
	\begin{equation}
	\label{f19_1}
	(1-\varepsilon\sign D_{\varkappa,\alpha})D_{\varkappa,\alpha}|x-y|^{-d-\alpha-\varkappa}\le\Delta_\varkappa g(t,\cdot,y)(x)\le(1+\varepsilon\sign D_{\varkappa,\alpha})D_{\varkappa,\alpha}|x-y|^{-d-\alpha-\varkappa}
\end{equation}for even integer $\varkappa\ge0$;
\begin{equation}
	\label{f20_1}
	\begin{split}
		\big(1-\varepsilon\sign(N_\varkappa(x-y,\nu))\big)N_{\varkappa,\alpha}\frac{(x-y,\nu)}{|x-y|^{d+\alpha+\varkappa+1}}\le i(\nabla_\varkappa g(t,\cdot,y)(x),\nu)\\
		\le\big(1+\varepsilon\sign(N_{\varkappa,\alpha}(x-y,\nu))\big)N_{\varkappa,\alpha}\frac{(x-y,\nu)}{|x-y|^{d+\alpha+\varkappa+1}}
	\end{split}
\end{equation}
for odd integer $\varkappa\ge1$ and any $\nu\in\mathbb{R}^d$.

\end{theorem}
\begin{proof}
	For the proof, one has to use relations \eqref{f10}, \eqref{f11} in addition to Theorems \ref{th1}~--- \ref{th4} statements.
\end{proof}
\begin{remark}
	For the differential operators under consideration, estimates \eqref{f19}~--- \eqref{f21} lead us to \eqref{f23} or \eqref{f23_1}, but the last one was obtained only for $\varkappa>0$ and it is the estimation from above. 
\end{remark}

\end{document}